\documentclass[a4paper,oneside,12pt]{article}

\usepackage{amsmath,amsfonts,amscd,amssymb}
\usepackage{longtable,geometry}
\usepackage[english]{babel}
\usepackage[utf8]{inputenc}
\usepackage[active]{srcltx}
\usepackage[T1]{fontenc}
\usepackage{graphicx}
\usepackage{pstricks}
\usepackage{bbm}
\usepackage{MnSymbol}
\usepackage{stmaryrd}
\usepackage{nicefrac}
\usepackage{calrsfs}
\newtheorem{theorem}{Theorem}[section]

\newtheorem{corollary}[theorem]{Corollary}
\newtheorem{lemma}[theorem]{Lemma}
\newtheorem{proposition}[theorem]{Proposition}

\newtheorem{remark}[theorem]{Remark}

\newcommand{\be}[1]{\begin{equation}\label{#1}}
\newcommand{\ee}{\end{equation}}
\numberwithin{equation}{section}

\newcommand{\ba}[1]{\begin{align}\label{#1}}
\newcommand{\ea}{\end{align}}
\numberwithin{equation}{section}

\newcommand{\ben}{\begin{equation*}}
\newcommand{\een}{\end{equation*}}
\numberwithin{equation}{section}

\newenvironment{proof}[1][\relax]%s
  {\paragraph{Proof\ifx#1\relax\else~of #1\fi}}%
  {~\hfill$\square$\par\bigskip}

\newcommand{\calA}{\mathcal{A}}
\newcommand{\calB}{\mathcal{B}}
\newcommand{\calC}{\mathcal{C}}

\newcommand{\calE}{\mathcal{E}}

\newcommand{\bbC}{\mathbb{C}}

\newcommand{\bbP}{\mathbb{P}}

\newcommand{\bbR}{\mathbb{R}}

\newcommand{\bbT}{\mathbb{T}}
\newcommand{\bbU}{\mathbb{U}}

\newcommand{\bbZ}{\mathbb{Z}}
\newcommand{\sfC}{{\sf C}}

\newcommand{\ep}{\varepsilon}

\newcommand{\rk}[1]{\bgroup\color{red}%
  \par\medskip\hrule\smallskip%
  \noindent\textbf{#1}%
  \par\smallskip\hrule\medskip\egroup}

\title{Limit of the Wulff Crystal when approaching criticality for site percolation on the triangular lattice}
\author{Hugo Duminil-Copin}
\date{\today}

\begin{document}
% Entete %
\maketitle
% Abstract %
\begin{abstract}
%Conformal invariance of a planar model strongly suggests that the Wulff crystal becomes a disk when approaching criticality. 
The understanding of site percolation on the triangular lattice progressed greatly in the last decade. Smirnov proved conformal invariance of critical percolation, thus paving the way for the construction of its scaling limit. Recently, the scaling limit of near-critical percolation was also constructed by Garban, Pete and Schramm. The aim of this very modest contribution is to explain how these results imply the convergence, as $p$ tends to $p_c$, of the Wulff crystal to a Euclidean disk. The main ingredient of the proof is the rotational invariance of the scaling limit of near-critical percolation proved by these three mathematicians.
\end{abstract}
% bibliography %
\section{Introduction}

\paragraph{Definition of the model}
Percolation as a physical model was introduced by Broadbent and
Hammersley in the fifties~\cite{BH57}.  For general background on percolation, we refer the reader to~\cite{Gri99,Kes82,BR06c}.

Let $\bbT$ be the regular triangular lattice given by the vertices $m+{\rm e}^{{\rm i} \pi/3}n$ where $m,n\in \bbZ$, and edges linking nearest neighbors together. In this article, the vertex set will be identified with the lattice itself.
For $p\in (0,1)$, \emph{site
  percolation} on $\bbT$ is defined as follows. The set of configurations is given by $\{{\rm open},{\rm closed}\}^\bbT$.  Each vertex, also called \emph{site}, is \emph{open} with probability $p$ and \emph{closed} otherwise,
independently of the state of other vertices. The probability measure thus obtained is denoted by $\bbP_p$.

%We are interested in the connectivity properties of the model. 
A {\em path} between $a$ and $b$ is a sequence of sites $v_0,\dots,v_k$ such that $v_0=a$ and $v_k=b$, and such that $v_iv_{i+1}$ is an edge of $\bbT$ for any $0\le i<k$. A path is said to be {\em open} if all its sites are open. Two
sites $a$ and $b$ of the triangular lattice are \emph{connected} (this is denoted by $a\longleftrightarrow b$) if there exists an open path between them. %The fact that $a$ belongs to an infinite path of open sites will be denoted by $0\longleftrightarrow \infty$. 
A {\em cluster} is a maximal connected graph for the relation $\longleftrightarrow$ on sites of $\bbT$.

\paragraph{The different phases} Bernoulli percolation undergoes a phase transition at $p_c=1/2$ (the corresponding result for bond percolation on the square lattice is due to Kesten \cite{Kes80}): in the {\em sub-critical phase} $p<p_c$, there is almost surely no infinite cluster, while in the {\em super-critical phase} $p>p_c$, there is almost surely a unique one.

The understanding of the {\em critical phase} $p=p_c$ has progressed greatly these last few years. In \cite{Smi01}, Smirnov proved Cardy's formula, thus providing the first rigorous proof of the conformal invariance of the model (see also \cite{Wer07,BD13} for details and references). This result led to many applications describing the critical phase. Among others, the convergence of interfaces was proved in \cite{CN06,CM07}, and critical exponents were computed in \cite{SW01}.

Another phase of interest is given by the so-called  {\em near-critical phase}. It is obtained by letting $p$ go to $p_c$ as a well-chosen function of the size of the system (see below for more details). This phase was first studied in the context of percolation by Kesten \cite{Kes87}, who used it to relate fractal properties of the critical phase to the behavior of the correlation length and the density of the infinite cluster (as $p$ tends to $p_c$). Recently, the scaling limit of near-critical percolation was proved to exist in \cite{GPSa}. This result will be instrumental in the proof of our main theorem.

\paragraph{Main statement} %Beyond the critical phase, physicists are interested in the behavior of the model through the phase transition. More precisely, they study the behavior of the percolation configuration of parameter $p$ as $p$ tends to $p_c$. 
Mathematicians and physicists are particularly interested in the following quantity, called the {\em correlation length}. For $p<p_c$ and for any $u$ on the unit circle $\bbU=\{z\in \bbC:|z|=1\}$, define
$$\tau_p(u)~:=~\left(\lim_{n\rightarrow \infty}-\tfrac{1}{n}\log \bbP_p(0\longleftrightarrow \widehat{nu})\right)^{-1},$$
where $\widehat{nu}$ is the site of $\bbT$ closest to $nu$.
In \cite{SW01}, the correlation length $\tau_p(u)$ was proved  to  behave like $(p_c-p)^{-4/3+o(1)}$ as $p\nearrow p_c$. 

Interestingly, conformal invariance at criticality is a strong indication that $\tau_p(u)$ becomes isotropic, meaning that it does not depend on $u\in \bbU$. The aim of this note is to show that this is indeed the case. 

Let $|\cdot|$ be the Euclidean norm on $\bbR^2$.
\begin{theorem}\label{thm:correlation}
For percolation on the triangular lattice,
 $\tau_p(u)/\tau_p(|u|)\longrightarrow 1$ uniformly in the direction $u\in \bbU$ as $p\nearrow p_c$.
\end{theorem} 

While this result is very intuitive once conformal invariance has been proved, it does not follow directly from it. More precisely, it requires some understanding of the near-critical phase mentioned above. %$\tau_p(u)$ (as $p\nearrow p_c$) and of $\theta(p)=\bbP_p(0\longleftrightarrow \infty)$ (as $p\searrow p_c$).
%Let us mention that for any $u$, $\bbP_{p_c}(0\longleftrightarrow \widehat{nu})/\bbP_{p_c}(0\longleftrightarrow \widehat{n|u|})$ can be shown to converge to 1 as $n$ tends to infinity; see \cite{GPS10a}.
The main input used in the proof is the spectacular and highly non-trivial result of \cite{GPSa}. %GPS10a}. 
In this paper, the scaling limit for near-critical percolation is proved to exist and to be invariant under rotations. This result constitutes the heart of the proof of Theorem~\ref{thm:correlation}, which then consists in connecting the correlation length to properties of this near-critical scaling limit (in some sense, the proof can be understood as an exchange of two limits).

\paragraph{Wulff crystal} Theorem~\ref{thm:correlation} has an interesting corollary. Consider the cluster $\sfC_0$ of the origin. When $p<p_c$, there exists a deterministic shape $W_p$ such that for any $\ep>0$, 
$$\bbP_p\left(\mathbf d_{\rm Hausdorff}\Big(\frac{\sfC_0}{\sqrt{n}},\frac{W_p}{\sqrt{{\rm Vol}(W_p)|}}\Big)>\ep~\Big| ~|\sfC_0|\ge n\right)\longrightarrow 0\text{\quad as $n\rightarrow \infty$,}$$
where ${\rm Vol}(E)$ denotes the volume of the set $E$, and $\mathbf d_{\rm Hausdorff}$ is the Hausdorff distance. In the previous formula, $W_p$ is the {\em Wulff crystal} defined by
$$W_p~:=~\{x\in \bbC:\langle x|u\rangle\le \tau_p(u),u\in \bbU\},$$
where $\langle \cdot|\cdot\rangle$ is the standard scalar product on $\bbC$. 

The Wulff crystal appears naturally when studying phase coexistence. Originally, the Wulff crystal was constructed rigorously in the context of the planar Ising model by Dobrushin, Koteck\'y and Shlosman \cite{DKS92} for very low temperature (see \cite{Pfi91,IS98} for extensions of this result). In the case of planar percolation, the first result is due to \cite{ACC90}. Let us mention that the Wulff construction was extended to higher dimensional percolation by Cerf \cite{Cer00} (see also \cite{Bod99,CP00} for the Ising case). We refer to \cite{Cer06} for a comprehensive exposition of the subject. 

The geometry of the Wulff crystal has been studied extensively since then. Let us mention that for any $p<p_c$, it is a strictly convex body with analytic boundary \cite{ACC90,Ale92,CI02}.

The expression of $W_p$ in terms of the correlation length, together with Theorem~\ref{thm:correlation}, implies the following result.
\begin{corollary}
When $p\nearrow p_c$, $\displaystyle W_p/\sqrt{{\rm Vol}(W_p)}$ tends to the disk $\{u\in \bbC:|u|\le 1\}$.
\end{corollary}
This corollary has a strong geometric interpretation. As $p\nearrow p_c$, the typical shape of a cluster conditioned to be large becomes round.

\paragraph{Super-critical phase} For the super-critical phase, the previous results can be translated in the following way. %The previous results extend to the super-critical phase. 
 For $p>p_c$, define
$$\tau^{\rm f}_p(u)^{-1}~:=~\lim_{n\rightarrow \infty}-\frac{1}{n}\log \bbP_p(0\longleftrightarrow \widehat{nu},0\not\longleftrightarrow \infty).$$
One can prove that $\tau^{\rm f}_p(u)=\frac12\tau_{1-p}(u)$; see \cite[Theorem A]{CIL10} for a much more precise (and much harder) result. This fact, together with Theorem~\ref{thm:correlation}, immediately implies that $\tau^{\rm f}_p(u)/\tau^{\rm f}_p(|u|)\rightarrow 1$, uniformly in the direction $u\in \bbU$, as $p\searrow p_c$.
The Wulff crystal can also be extended to the super-critical phase. %In such case, consider the cluster of the origin $\sfC_0$ conditioned to have a volume larger than $n$ but finite. 
When $p>p_c$, we find that for any $\ep>0$,
$$\bbP_0\left(\mathbf d_{\rm Hausdorff}\Big(\frac{\sfC_0}{\sqrt{n}},\frac{W_{1-p}}{\sqrt{{\rm Vol}(W_{1-p})}}\Big)>\ep~\Big|~n\le |\sfC_0|<\infty\right)\longrightarrow 0\text{\quad as $n\rightarrow \infty$.}$$
%This yields that the typical shape of a large but finite cluster in super-critical phase becomes circular as $p\searrow p_c$.
\paragraph{Other models} Let us mention that conformal invariance has been proved for a number of models, including the dimer model \cite{Ken00} and the Ising model \cite{Smi10,CS09}; see \cite{DCS11} for lecture notes on the subject. In both cases, exact computations (see \cite{McCW73} for the Ising model) allow one to show that the correlation length becomes isotropic, hence providing an extension of Theorem~\ref{thm:correlation}. For the Ising model, we refer to \cite{McCW73}  for the original computation, and to  \cite{BDCS11,Dum12} for a recent computation.  For percolation, no exact computation is available and the passage via the near-critical regime seems required. For the Ising model, the near-critical phase was also studied in \cite{DGP12}.

\paragraph{An open question} To conclude, let us mention the following question, which was asked by I. Benjamini: let $p>p_c$ and condition 0 to be connected to infinity. Consider the sequence of balls of center 0 and radius $n$ for the graph distance on the infinite cluster. Show that these balls possess a limiting shape $U_p$ which becomes round as $p\searrow p_c$.

\paragraph{Notation}
When points are considered as elements of the plane, we usually use complex numbers to position them.
When they are considered as vertices of $\bbT$, we prefer oblique coordinates.  More precisely, the point $(m,n)\in\bbZ^2$ is the point $m+{\rm e}^{{\rm i} \pi/3}n$. 
 It shall be clear from the context whether complex or oblique coordinates are used. The set $[m_1,m_2]\times [n_1,n_2]$ is therefore the parallelogram composed of sites $m+{\rm e}^{{\rm i} \pi/3}n$ with $m_1\le m\le m_2$ and $n_1\le n\le n_2$. 
  
Define $\calC_{\rm hor}([m_1,m_2]\times [n_1,n_2])$ to be the event that there exits an open path included in $[m_1,m_2]\times [n_1,n_2]$ from $\{m_1\}\times[n_1,n_2]$ to $\{m_2\}\times[n_1,n_2]$. If the event occurs, the parallelogram is said to be {\em crossed}. Let $\calC_{\rm circuit}(x,n,2n)$ be the event that there exists an open circuit (meaning a path starting and ending at the same site) in $(x+[-2n,2n]^2)\setminus(x+[-n,n]^2)$ surrounding the origin.

\subsection{Two important inputs}
We will use tools of percolation theory such as correlation inequalities (in our case the FKG and BK inequalities) and Russo's formula. We shall not remind these classical facts here, and the reader is referred to \cite{Gri99} for precise definitions. In addition to these facts, we will harness two important results. The first one relates crossing probabilities for parallelograms of different shape. It is usually referred to as the Russo-Seymour-Welsh theorem \cite{Rus78,SW78}, or simply RSW. %We will be using standard tools of percolation together with recent results on the near-critical phase of site percolation. 
%\paragraph{Correlation inequalities} First, recall that an event $A\subseteq\{0,1\}^{\bbZ^2}$ is \emph{increasing} if $\omega\in A$ and $\omega\le \omega'$ imply $\omega'\in A$. Two important correlation inequalities related to increasing events will be used in the article. 
%The first inequality is the so-called \emph{Harris inequality}. It is a particular case of the FKG inequality \cite{FKG71}. For two increasing events $A$ and $B$,
%$$\bbP_p(A\cap B)\geq \bbP_p(A)\bbP_p(B).$$
%The second inequality is the \emph{BK inequality}. For $A$ and $B$ two increasing events, the event $A\circ B$ is defined as follows. A configuration $\omega\in \{0,1\}^{\mathbb Z^2}$ belongs to $A\circ B$ if there exists a set $F=F(\omega)$ so that $\mathbbm{1}_F\omega\in A$ and $\mathbbm{1}_{F^c}\omega\in B$. In this case, $A$ and $B$ are said to occur {\em disjointly}. The disjoint occurrence of $k$ increasing events $A_1,\ldots,A_k$ can be defined by
%$$A_1\circ\cdots\circ A_k=A_1\circ(A_2\circ\cdots(A_{k-1}\circ A_k)).$$
%Then, for any increasing events $A_1,\dots,A_k$ depending on a finite number of sites,
%$$\bbP_p(A_1\circ\cdots\circ A_k)\leq \bbP_p(A_1)\cdots\bbP_p(A_k).$$ 
%We refer the reader to the book \cite{Gri99} for proofs of these two inequalities.

\begin{theorem}[Russo-Seymour-Welsh]\label{main theorem}There exist $p_0>0$ and $c_1,c_2>0$ verifying for every $n,m$ and $p\in[p_0,1-p_0]$,
\begin{align}\bbP_p\big[\calC_{\rm hor}([0,m]\times[0,2n])\big] \le c_1\Big(\bbP_p\big[\calC_{\rm hor}([0,m]\times[0,n])\big]\Big)^{c_2}.\end{align}
\end{theorem}

Note that we do not require that the parameter $p\in[0,1]$ is critical, but rather that it is bounded away from 0 and 1. For a proof of this version of the theorem on $\bbT$, we refer to \cite{Nol08}.
\bigbreak
Another important result, due to Garban, Pete and Schramm, will be required. Let $\calA_4(1,n)$ be the event that there exist four disjoint paths from neighbors of the origin to $\bbT\setminus [-n,n]^2$, indexed in the clockwise order by $\gamma_1,\gamma_2,\gamma_3$ and $\gamma_4$, with the property that $\gamma_1$ and $\gamma_3$ are open, while $\gamma_2$ and $\gamma_4$ are closed (meaning that they contain closed sites only). 
We set 
$$r(n)=\frac{\bbP_{p_c}[\calA_4(1,n)]}{n^2}.$$ 

For $r>0$ and $u\in\bbR^2$, let $\calB_r(u)=\{z\in \bbR^2:|z-u|\le r\}$ be the Euclidean ball of radius $r$ around $u$. For two sets $A$ and $B$ in $\bbR^2$, we say that $A\longleftrightarrow B$ if there exists $a\in A\cap \bbT$ and $b\in B\cap \bbT$ such that $a\longleftrightarrow b$.

\begin{proposition}\label{prop:inv}There exists $f:\bbR\times \bbR_+\rightarrow [0,1]$ such that for any $u\in \bbR^2$,
\begin{eqnarray}\label{eq:32} \bbP_{p_c-\lambda r(n)}\big[\calB_n(0)\longleftrightarrow \calB_n(nu)\big]&\longrightarrow &f(\lambda,|u|)
\end{eqnarray}
as $n$ tends to $\infty$. %In the formula, $|\cdot|$ denotes the complex modulus.
\end{proposition}

The previous proposition has two interesting features. First, the quantity on the left possesses a limit as $n$ tends to infinity. Second, this limit is invariant under rotations.
For $\lambda=0$, the result follows from the convergence of percolation interfaces to $\mathsf{CLE}(6)$ \cite{CN06}. For more general values of $\lambda$, the result is much more difficult. Let us briefly explain where this proposition comes from.

The scaling limit described in \cite{GPS10a,GPSa} is a limit, in the sense of the Quad-topology introduced in \cite{SS11}, of percolation configurations $\bbP_{p_c-\lambda r(n)}$ on $\frac1n \bbT$. %This topology is the smallest making maps $\delta_Q$ continuous, where $Q$ is a smooth topological rectangle, and  $\delta_Q$ is the map asserting if $Q$ possesses an open crossing or not. 
This topology is sufficiently strong to control events considered in Proposition~\ref{prop:inv}. The existence of the scaling limit is justified by a careful study of macroscopic ``pivotal points''. This existence implies the existence of the limit in \eqref{eq:32}. 
%The pivotal points rule the macroscopic behavior of critical percolation, and therefore the behavior of its scaling limit. 
The scale $r(n)$ corresponds to the scale for which a variation of $\lambda r(n)$ will alter the pivotal points, and therefore the scaling limit, but not too drastically. This fact enabled Garban, Pete and Schramm to construct the scaling limit of near-critical percolation from the scaling limit of critical percolation. The invariance under rotation of the near-critical scaling limit is then a consequence of the invariance under rotation of the critical one. We refer to \cite{GPSa} for more details.

In the proof, the near-critical phase will be used at its full strength. On the one hand, the scaling limit is still invariant under rotations. On the other hand, as $\lambda\rightarrow \infty$, the ``crossing probabilities'' tend to 0. The existence of such a phase is crucial in our proof.

\section{Proof of Theorem~\ref{main theorem}}

The proof consists in estimating the correlation length $\tau_p(u)$ using $f(\lambda,|u|)$. It is known since \cite{Kes87} that the correlation length is related to crossing probabilities. Yet, previous studies were interested in relations which are only valid up to bounded multiplicative constants. Here, we will need a slightly better control (roughly speaking that these constants tend to 1 as $p$ goes to $p_c$). 

In order to relate $\tau_p(u)$ and $f(\lambda,|u|)$, we use the existence of different parameters $\delta,\lambda, p_0$ and $L_p$ with some specific properties presented in the next proposition.
\begin{proposition}\label{conditions}
Let $\ep>0$. There exist $\lambda,\delta>0$ and $p_0<p_c$ such that for any $p\in[p_0,p_c]$, there exists $L_p\ge 0$ with the following three properties:
\begin{enumerate}
\item[{\rm P1}] \quad for any $\theta\in[0,2\pi)$,
$$\displaystyle f(\lambda,\delta^{-1})^{1+\ep}~\le~ \bbP_{p}\big[\calB_{\delta L_p}(0)\longleftrightarrow \calB_{\delta L_p}(L_pe^{i\theta})\big]~\le~ f(\lambda,\delta^{-1})^{1-\ep},$$
\item[{\rm P2}] \quad$\displaystyle \bbP_{p}\big[\calC_{\rm circuit}(0,\delta L_p,2\delta L_p)\big]~\ge~ f(\lambda,\delta^{-1})^{\ep},$
\item[{\rm P3}] \quad$\displaystyle  \delta ~\ge~ 2f(\lambda,\delta^{-1})^{\ep}$.
\end{enumerate}
\end{proposition}
%The scale $L_p$ will be chosen of the order of $\tau_p(|u|)$. 
The main part of the proof of Theorem~\ref{main theorem} will be to show Proposition~\ref{conditions}, i.e. that for any $\ep>0$, the constants $\lambda,\delta, p_0$ and $L_p$ can indeed be constructed. Before proving Proposition~\ref{conditions}, let us show how it implies Theorem~\ref{main theorem}.
\begin{proof}[Theorem~\ref{main theorem}]
Fix $\ep>0$. Define $\lambda,\delta>0$ and $p_0<p_c$ such that Proposition~\ref{conditions} holds true. Let $\theta\in[0,2\pi)$ and $p_0<p<p_c$. Consider $L_p$ defined as in the proposition.

For $K\ge 1$, consider the following three events:
\begin{enumerate}
\item[$\calE_1$=]``$\calB_{\delta L_p}(0)$ and $\calB_{\delta L_p}(KL_pe^{i\theta})$ are full '',
\item[$\calE_2$=]`` $\calB_{\delta L_p}(kL_p{\rm e}^{i\theta})\longleftrightarrow \calB_{\delta L_p}((k+1)L_p{\rm e}^{i\theta})$ for every $0\le k\le K$ '',
\item[$\calE_3$=]`` $\calC_{\rm circuit}(kL_p{\rm e}^{i\theta},\delta L_p,2\delta L_p)$ for every $0\le k\le K$ ''.
\end{enumerate}
As shown on Fig.~\ref{fig:events}, if all these events occur, then $0$ and the site of $\bbT$ closest to $KL_p{\rm e}^{i\theta}$, denoted $\widehat{KL_p{\rm e}^{i\theta}}$, are connected by an open path. The FKG inequality (see \cite[Theorem 2.4]{Gri99}) implies that \begin{align*}\bbP_{p}\big[0\longleftrightarrow \widehat{KL_pe^{i\theta}}\big]&\ge \bbP_{p}\left[\calE_1\right]\bbP_{p}\left[\calE_2\right]\bbP_{p}\left[\calE_3\right]\\
&\ge p^{8(\delta L_p)^2}\cdot f(\lambda,\delta^{-1})^{(1+\ep)K}\cdot f(\lambda,\delta^{-1})^{\ep K}.\end{align*}
We have used P1 and P2 to bound the probabilities of $\calE_2$ and $\calE_3$ in the second inequality. The bound on $\bbP_p[\calE_1]$ comes from the fact that there are less than $8(\delta L_p)^2$ sites in $\calB_{\delta L_p}(0)\cup\calB_{\delta L_p}(KL_pe^{i\theta})$.

By taking the logarithm and letting $K$ tend to infinity, we obtain that for $p_0<p<p_c$,
\be{lower}\frac1{\tau_{p}(e^{i\theta})}\le -(1+2\ep)\frac{\log f(\lambda,\delta^{-1})}{L_p}.\ee
This provides us with an upper bound that we will match with the lower bound below. 
\bigbreak
Let us now turn to the lower bound. Assume that $0$ and $\widehat{KL_p{\rm e}^{i\theta}}$ are connected. Let $0\le N\le 2\delta^{-1}$ such that $\sin (\frac{2\pi}{N})<\delta/2$. Define
$$\Theta=\big\{L_p,L_pe^{\frac{2\pi i}{N} },L_pe^{2\frac{2\pi i}{N}} ,\dots, L_pe^{(N-1)\frac{2\pi i}{N}} \big\}.$$ 
We claim that if $0$ and $KL_p{\rm e}^{i\theta}$ are connected, then there must exist a sequence of sites $0=x_0,x_1,\dots,x_K$ such that $x_{i+1}-x_i\in \Theta$ and $\calB_{\delta L_p}(x_i)\longleftrightarrow \calB_{\delta L_p}(x_{i+1})$ occurs for every $0\le i<K$. Furthermore, the events $\calB_{\delta L_p}(x_i)\longleftrightarrow \calB_{\delta L_p}(x_{i+1})$ occur disjointly in the sense of \cite[Section 2.3]{Gri99}.

In order to prove this claim, consider a self-avoiding open path $\gamma=(\gamma_i)_{0\le i\le r}$ from $0$ to $\widehat{KL_p{\rm e}^{i\theta}}$. Let $y_1$ be the first point of this path which is outside of the Euclidean ball of radius $L_p$ around 0. Define $x_1\in \Theta$ such that $|y_1-x_1|\le \delta L_p$. The choice of $N$ guarantees the existence of $x_1$. Let $y_2$ be the first point of $\gamma[y_1,r]$ outside of the Euclidean ball of radius $L_p$ around $x_1$. We pick $x_2$ such that $x_2-x_1\in \Theta$ and $|y_2-x_2|\le \delta L_p$. We construct $(x_i)_{0\le i\le K}$ iteratively. See Fig.~\ref{fig:events} for an illustration. By construction, the events occur {\em disjointly}  since the path $\gamma$ is self-avoiding. We note $A\circ B$ for the disjoint occurrence (see \cite[Theorem 2.12]{Gri99}). The union bound and the BK inequality give
\begin{align*}\bbP_p[0\longleftrightarrow \widehat{KL_pe^{i\theta}})]&\le \sum_{(x_i)_{i\le K}}\bbP_p\Big[\big\{\calB_{\delta L_p}(0)\longleftrightarrow \calB_{\delta L_p}(x_{1})\big\}\circ\cdots\\
&\quad\quad\quad\quad\quad\quad\quad\quad \cdots\circ\big\{\calB_{\delta L_p}(x_{j-1})\longleftrightarrow \calB_{\delta L_p}(x_{j})\big\}\Big]\\
&\le\left(\frac{2}{\delta}\right)^K\bbP_p[\calB_{\delta L_p}(0)\longleftrightarrow \calB_{\delta L_p}(x_{1})]^K\\
&\le\big(\frac2{\delta}f(\lambda,\delta^{-1})^{1-\ep}\big)^K\le f(\lambda,\delta^{-1})^{(1-2\ep)K}.\end{align*}
In the second inequality, we used the fact that the cardinality of $\Theta$ is bounded by $2\delta^{-1}$. In the last line, we used P1 and then P3. By taking the logarithm and letting $K$ go to infinity, we obtain that for $p_0<p<p_c$\be{upper}\frac1{\tau_{p}(e^{i\theta})}\ge -(1-2\ep)\frac{\log f(\lambda,\delta^{-1})}{L_p}.\ee
Therefore, for every $p_0<p<p_c$, \eqref{lower} and \eqref{upper} imply that $1-4\ep<\tau_p(e^{i\theta})/\tau_p(1)\le 1+5\ep$ for any $\theta\in[0,2\pi)$.\end{proof}

\begin{figure}

\begin{center}
\includegraphics[width=1.00\textwidth]{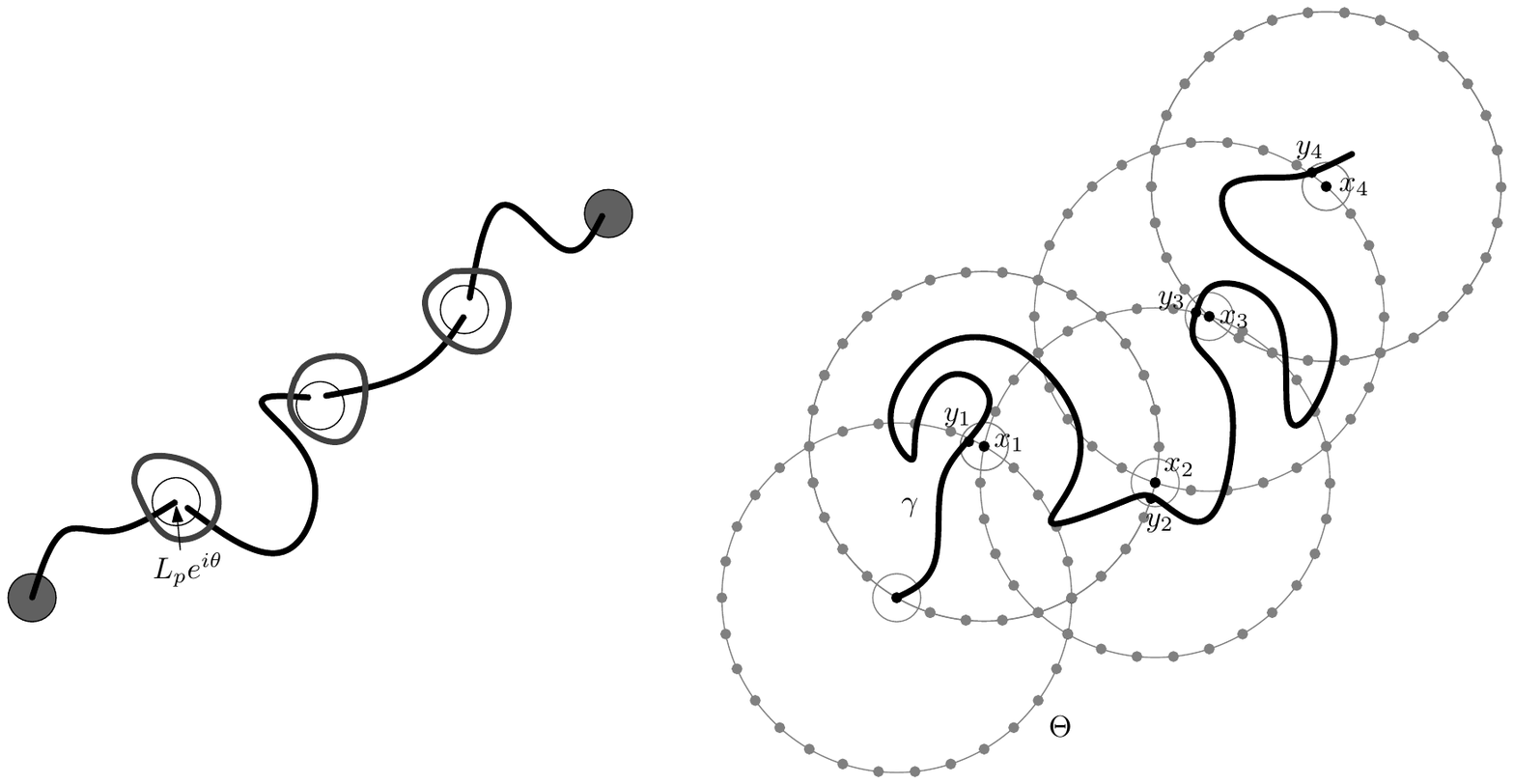}\end{center}
\caption{\label{fig:events}On the left, the events $\calE_1$, $\calE_2$ and $\calE_3$. On the right, the construction corresponding to the upper bound.}
\end{figure}

We now prove Proposition~\ref{conditions}. Property P1 will be guaranteed by the fact that we set $L_p$ in such a way that 
\begin{equation}\label{eq:def}p=p_c-\lambda r(\delta L_p)\end{equation} (we should be careful about the rounding operation since $r$ takes only discrete values, but this is easily shown to be irrelevant). P1 therefore follows from Proposition~\ref{prop:inv} and any choice of $\delta,\lambda$ by setting $L_p$ as in \eqref{eq:def}. The only mild condition is that $L_p$ is large enough, or equivalently that $p$ is close enough to $p_c$. 

From now on, we will assume that $L_p$ is chosen as in \eqref{eq:def}. We need to explain how to choose $\delta$ and $\lambda$ in such a way that P2 and P3 are also satisfied.
Since properties P2 and P3 are conditions on the limit $f(\lambda,\delta^{-1})$, the proof boils down to finding $\lambda$ and $\delta$ in such a way that $f(\lambda,\delta^{-1})$ is small enough compared to $\bbP_{p}\big[\calC_{\rm circuit}(\delta L_p,2\delta L_p)\big]$ and $\delta$. 
\begin{remark} Taking $\delta$ small enough is not sufficient since in such case, $f(\lambda,\delta^{-1})$ could a priori become larger than $\delta/2$. This is what is happening when $\lambda=0$ since this quantity tends to zero as $\delta^{10/48}$. We will prove that for $\lambda$ large enough, $f(\lambda,\delta^{-1})$ tends to zero (as $\delta$ tends to zero) exponential fast in $\delta^{-1}$.%Therefore, a slightly more careful study of these quantities must be performed.
\end{remark}
Let us start by two simple lemmata. 
\begin{lemma}\label{lem:1}
  Fix $p\in (0,1)$ and $n,k>0$. Then,
  $$28 \bbP_p\big[\calC_{\rm hor}([0,2^{k}n]\times[0,2^{k+1}n])\big]~\le~ \Big(28\,\bbP_p\big[\calC_{\rm hor}([0,n]\times[0,2n])\big]\Big)^{2^{k}}.$$\end{lemma}

Even though this fact is classical, the proof is elegant and short therefore we choose to include it here.
\begin{proof}
  Let $n>0$ and consider the parallelograms $P_1,\dots,P_8$ defined by
  \begin{center}
  \begin{tabular}{ll}
  $P_1=[0,n]\times[0,2n]$,\quad\quad & $P_2=[0,n]\times[n,3n]$, \\
  $P_3=[0,n]\times[2n,4n]$,\quad\quad & $P_4=[n,2n]\times[0,2n]$,\\
  $P_5=[n,2n]\times[n,3n]$,\quad\quad & $P_6=[n,2n]\times[2n,4n]$,\\
  $P_7=[0,2n]\times[n,2n]$,\quad\quad & $P_8=[0,2n]\times[2n,3n]$.
  \end{tabular}
  \end{center}
  These parallelograms have the property that whenever
  there exists an open path from $\{0\}\times[0,4n]$ to $\{2n\}\times[0,4n]$, then there exist (at least) two disjoint paths crossing parallelograms $P_i$. In other words, the event $\calC_{\rm hor}([0,2n]\times[0,4n])$ is included in the event that two of the eight parallelograms $P_i$ contain an open path crossing in the easy direction (meaning vertically if the base is larger than the height, and horizontally if the reverse is true), and that these events occur disjointly.
The BK inequality (see \cite[Theorem 2.12]{Gri99}) implies that 
  \begin{align*}
    \bbP_p\big[\calC_{\rm hor}([0,2n]\times[0,4n])\big] & \le
    \binom 8 2~\bbP_p\big[\calC_{\rm hor}([0,n]\times[0,2n])\big]^2\\
    &=28\,\bbP_p\big[\calC_{\rm hor}([0,n]\times[0,2n])\big]^2.
  \end{align*}
The lemma follows by applying this inequality iteratively to integers of the form $2^jn$ for $0\le j\le k-1$.
\end{proof}

The next lemma is a standard application of the RSW theorem. We do not remind the proof here.
\begin{lemma}\label{lem:4}
There exist $0<p_0<p_c$ and $c_3,c_4>0$ such that for any $p\in[p_0,1-p_0]$ and $n\ge 0$,
$$\bbP_{p}\big[\calC_{\rm circuit}(n,2n)\big]\ge c_3\,\bbP_p\big[\calC_{\rm hor}([0,n]\times[0,2n])\big]^{c_4}.$$
\end{lemma}

The previous lemma provides us with a lower bound on $\bbP_{p}\big[\calC_{\rm circuit}(n,2n)\big]$ in terms of $\bbP_p\big[\calC_{\rm hor}([0,n]\times[0,2n])\big]$, while the next lemma provides us with an upper bound on $\bbP_{p}\big[\calB_n(0)\longleftrightarrow \calB_n(x)\big]$ in terms of the same quantity.
\begin{lemma}\label{lem:2}
There exist $0<p_0<p_c$ and $c_5,c_6>0$ such that for any $p\in[p_0,1-p_0]$, for any $n \ge 0$ and $\delta>0$,
$$\bbP_{p}\big[\calB_n(0)\longleftrightarrow \calB_n(n/\delta)\big]\le \Big(c_5\,\bbP_p\big[\calC_{\rm hor}([0,n]\times[0,2n])\big]\Big)^{c_6/\delta}.$$\end{lemma}

\begin{proof}
For $n<m$, let $\calC_{\rm in/out}(n,m)$ be the event that there exists an open path between $[-n,n]^2$ and $\bbT\setminus [-m,m]^2$.

The RSW theorem implies that
$$\bbP_p\big[\calC_{\rm hor}([0,n]\times[0,4n])\big]\le c_1\bbP_p\big[\calC_{\rm hor}([0,n]\times[0,2n])\big]^{c_2}.$$
Since one of the four parallelograms $[-2n,2n]\times[-2n,-n]$, $[-2n,2n]\times[n,2n]$, $[-2n,-n]\times[-2n,2n]$ and $[n,2n]\times[-2n,2n]$ must be crossed if the annulus $[-2n,2n]^2\setminus[-n,n]^2$ contains an open path from the interior to the exterior, we obtain 
$$\bbP_p\big[\calC_{\rm in/out}(n,2n)\big]\le 4c_1\bbP_p\big[\calC_{\rm hor}([0,n]\times[0,2n])\big]^{c_2}.$$
We find immediately that for any $k>0$,
\begin{align*}\bbP_p\big[\calC_{\rm in/out}(n,2^kn)\big]&\le \prod_{j=0}^{k-1} \bbP_p\big[\calC_{\rm in/out}(2^jn,2^{j+1}n)\big]\\
&\le \prod_{j=0}^{k-1}4c_1\bbP_p\big[\calC_{\rm hor}([0,2^jn]\times[0,2^{j+1}n])\big]^{c_2}\\
&\le \prod_{j=0}^{k-1}\frac{4c_1}{28}\Big(28\bbP_p\big[\calC_{\rm hor}([0,n]\times[0,2n])\big]\big)^{2^j}\Big)^{c_2}\\
&\le \Big(c_7\,\bbP_p\big[\calC_{\rm hor}([0,n]\times[0,2n])\big]\Big)^{c_22^k}\end{align*}
where $c_7>0$ is large enough. We have used independence in the first inequality. In the third inequality, we used Lemma~\ref{lem:1}.
\medbreak
Let $2^k\le \delta^{-1}<2^{k+1}$.  Since 
$$\big[-2^{k-2}n,2^{k-2}n\big]^2\setminus\big[-n,n\big]^2\subset \calB_{n/(2\delta)}(0)\setminus \calB_n(0),$$ we immediately get
\begin{align*}\bbP_p\big[\calB_n(0)\longleftrightarrow \calB_n(n/\delta)\big]&\le \bbP_p\big[\calC_{\rm in/out}(n,2^{k-2}n)\big]\\
&\le \Big(c_7\,\bbP_p\big[\calC_{\rm hor}([0,n]\times[0,2n])\big]\Big)^{c_22^{k-2}}\\
&\le \Big(c_7\,\bbP_p\big[\calC_{\rm hor}([0,n]\times[0,2n])\big]\Big)^{8c_2/\delta}\end{align*}
and the claim follows.\end{proof}

Now, fix $\mu=\mu(c_3,c_4,c_5)>0$ small enough that $c_5\mu\le \mu^{1/2}\le 1/2$ and $\mu^{c_4}\le c_3$.
For any $\ep>0$, the two previous lemmata show the existence of $\delta=\delta(\mu,\ep)>0$ such that  $$\bbP_p\big[\calC_{\rm hor}([0,n]\times[0,2n])\big]\le \mu$$ implies that
$$\begin{cases}\delta\ge \bbP_p\big[\calB_n(0)\longleftrightarrow \calB_n(n/\delta)\big]^{\ep/(1-\ep)},&\,\\
\bbP_{p}\big[\calC_{\rm circuit}(n,2n)\big]\ge\bbP_p\big[\calB_n(0)\longleftrightarrow \calB_n(n/\delta)\big]^{\ep/(1-\ep)}.&\, \end{cases}$$
In particular, if the previous implication can be applied with $p$ and $L_p$ such that $p=p_c-\lambda r(\delta L_p)$ (for a well-chosen $\lambda$), then P1, P2 and P3 are satisfied and Proposition~\ref{conditions} is proved. Therefore, it remains to prove that there exists $\lambda=\lambda(\mu,\delta)>0$ such that
$$\bbP_p\big[\calC_{\rm hor}([0,\delta L_p]\times[0,2\delta L_p])\big]\le \mu$$
for $p$ close enough to $p_c$, or equivalently, such that
\begin{equation}\label{ert} \limsup_{n\rightarrow \infty}\,\bbP_{p_c-\lambda r(n)}\big[\calC_{\rm hor}([0,n]\times[0,2n])\big]<\mu\end{equation}
for $n$ large enough.
This is the object of our last lemma, which concludes the proof of Proposition~\ref{conditions}. 

\begin{lemma}\label{lem:8}
Let $\mu>0$. There exists $\lambda>0$ such that,
\begin{equation} \limsup_{n\rightarrow \infty}\,\bbP_{p_c-\lambda r(n)}\big[\calC_{\rm hor}([0,n]\times[0,2n])\big]<\mu.\end{equation}
\end{lemma}
Note that the $\limsup$ above is in fact a limit by \cite{GPSa}, but this fact is of no relevance here. While the proof is fairly easy, we justify it in details. The main ingredient is the quasi-multiplicativity property for near-critical percolation.

\begin{proof}
Fix $\mu>0$ and let $\lambda>0$ to be fixed later. Fix $0<\eta<1/28$ and $p\in(p_0,1-p_0)$ with $p_0$ defined in Lemma~\ref{lem:2}. If
 $$\bbP_{p}\big[\calC_{\rm hor}([0,m]\times[0,2m])\big]<\eta,$$ 
 then Lemma~\ref{lem:1} implies that for any $k\ge 1$, 
$$\bbP_{p}\big[\calC_{\rm hor}([0,2^{k}n]\times[0,2^{k+1}n])\big]\le \frac1{28}\Big(28\eta\Big)^{2^{k}}\le \eta.$$
By RSW, we get that for $2^k n\le m\le 2^{k+1}n$,
\begin{align*}\bbP_{p}\big[\calC_{\rm hor}([0,m]\times[0,2m])\big]&\le \bbP_{p}\big[\calC_{\rm hor}([0,2^kn]\times[0,2^{k+2}n])\big]\\
&\le c_1\bbP_{p}\big[\calC_{\rm hor}([0,2^kn]\times[0,2^{k+1}n])\big]^{c_2}\\
&\le c_1\eta^{c_2}.\end{align*}
In conclusion, if $0<\eta\ll \mu$ are chosen in such a way that $\mu>c_1\eta^{c_2}$ and $\eta<1/28$, we find that
$\bbP_{p}\big[\calC_{\rm hor}([0,n]\times[0,2n])\big]<\eta$ implies that $\bbP_{p}\big[\calC_{\rm hor}([0,m]\times[0,2m])\big]<\mu$ for any $m\ge n$.
\medbreak

Define $$L_\eta(p)=\inf\Big\{n\ge 0:\bbP_p\big[\calC_{\rm hor}([0,n]\times[0,2n])\big]<\eta\Big\}.$$
%The quantity $L_\eta(p)$ is sometimes called the (geometric) correlation length of the model. 
%Fix $\lambda,\mu>0$ and $\eta=\eta(\mu)$ defined as before. 
The assumption that $\bbP_{p_c-\lambda r(n)}\big[\calC_{\rm hor}([0,m]\times[0,2m])\big]\ge \mu$ boils down to the assumption that $L_\eta(p)\ge n$ for every $p_c-\lambda r(n)<p<p_c$. We use this formulation from now on in order to bound the derivative of $\bbP_p\big[\calC_{\rm hor}([0,n]\times[0,2n])\big]$ from below. 

Russo's formula implies that 
\begin{align}\label{Russo}\frac{\rm d}{{\rm d}p}\bbP_p&\big[\calC_{\rm hor}([0,n]\times[0,2n])\big]\\
&=\sum_{v\in[0,n]\times[0,2n]}\bbP_p\big[v\text{ is pivotal for }\calC_{\rm hor}([0,n]\times[0,2n])\big],\nonumber\end{align}
where $v$ is pivotal for a configuration $\omega$ if the following is true: $\omega\in A$ if the site $v$ is switched to open and $\omega\notin A$ if it is switched to closed. In our case, $v$ is pivotal if there exist four disjoint arms, two open ones going from neighbors of $v$ to the left and right sides of $[0,n]\times[0,2n]$, and two closed ones going from neighbors of $v$ to the top and bottom sides of $[0,n]\times[0,2n]$. Since $n\le L_\eta(p)$, classical properties of arm-events (namely the extendability and quasimultiplicativity, see \cite[Propositions~15 and 16]{Nol08}) imply the existence of $c_8=c_8(\eta)>0$ such that
\begin{equation}\label{first bound}\bbP_{p}[v\text{ is pivotal for }\calC_{\rm hor}([0,n]\times[0,2n])]\ge c_8 \bbP_{p}[\calA_4(1,n)]\end{equation}
for any $v\in [\frac n3,\frac{2n}3]\times[\frac n2,n]$ and for any $p>p_c-\lambda r(n)$.
It is also well-known that under $L_\eta(p)$, arm-exponents do not vary (see e.g. \cite[Theorem~26]{Nol08} for the case of the triangular lattice), so that there exists $c_9=c_9(\eta)>0$ such that
\begin{equation}\label{second bound}\bbP_{p}[\calA_4(1,n)]\ge c_9 \bbP_{p_c}[\calA_4(1,n)]=c_9 \frac{1}{n^2r(n)}\end{equation}
for any $p_c>p>p_c-\lambda r(n)$.
Putting \eqref{Russo}, \eqref{first bound} and \eqref{second bound} together, we obtain that
$$\frac{\rm d}{{\rm d}p}\bbP_p\big[\calC_{\rm hor}([0,n]\times[0,2n])\big]\ge c_8c_9\sum_{v\in [\frac n3,\frac{2n}3]\times[\frac n2,n]} \frac{1}{n^2r(n)}= \frac{c_8c_9}{6r(n)}.$$
Integrating this inequality between $p_c-\lambda r(n)$ and $p_c$, we find that 
$$\bbP_{p_c-\lambda r(n)}\big[\calC_{\rm hor}([0,n]\times[0,2n])\big]\le 1-\frac{c_8c_9\lambda}{6 r(n)}.$$
Since $c_8$ and $c_9$ depend only on $\mu$ (since they are functions of $\eta$), we can finally conclude that for $\lambda>\frac{6}{c_8c_9}$,
$$\bbP_{p_c-\lambda r(n)}\big[\calC_{\rm hor}([0,n]\times[0,2n])\big]\le \mu.$$
\end{proof}

\paragraph{Acknowledgements.} The author was supported by the ERC AG CONFRA, as well as by the Swiss
{FNS}. The author would like to thank Nicolas Curien for a very interesting and useful discussion, and for comments on the manuscript. The author would also like to thank Itai Benjamini, Christophe Garban, G\'abor Pete and Alan Hammond for stimulating discussions.

\bibliographystyle{amsalpha}
\bibliography{bibli}

\def\cprime{$'$}
\providecommand{\bysame}{\leavevmode\hbox to3em{\hrulefill}\thinspace}
\providecommand{\MR}{\relax\ifhmode\unskip\space\fi MR }
\providecommand{\MRhref}[2]{%
  \href{http://www.ams.org/mathscinet-getitem?mr=#1}{#2}
}
\providecommand{\href}[2]{#2}
\begin{thebibliography}{DCGP11}

\bibitem[ACC90]{ACC90}
K.~Alexander, J.~T. Chayes, and L.~Chayes, \emph{The {W}ulff construction and
  asymptotics of the finite cluster distribution for two-dimensional
  {B}ernoulli percolation}, Comm. Math. Phys. \textbf{131} (1990), no.~1,
  1--50.

\bibitem[Ale92]{Ale92}
Kenneth~S. Alexander, \emph{Stability of the {W}ulff minimum and fluctuations
  in shape for large finite clusters in two-dimensional percolation}, Probab.
  Theory Related Fields \textbf{91} (1992), no.~3-4, 507--532.

\bibitem[BDC12]{BDCS11}
V.~Beffara and H.~Duminil-Copin, \emph{Smirnov's fermionic observable away from
  criticality}, Ann. Probab. \textbf{40} (2012), no.~6, 2667--2689.

\bibitem[BDC13]{BD13}
\bysame, \emph{Lectures on planar percolation with a glimpse of {S}chramm
  {L}oewner {E}volution}, arXiv:1107.0158 (2013), 43 pages.

\bibitem[BH57]{BH57}
S.~R. Broadbent and J.M. Hammersley, \emph{Percolation processes i. crystals
  and mazes}, Math. Proceedings of the Cambridge Philosophical Society
  \textbf{53} (1957), no.~03, 629--641.

\bibitem[Bod99]{Bod99}
T.~Bodineau, \emph{The {W}ulff construction in three and more dimensions},
  Comm. Math. Phys. \textbf{207} (1999), no.~1, 197--229.

\bibitem[BR06]{BR06c}
B.~Bollob{\'a}s and O.~Riordan, \emph{Percolation}, Cambridge Univ Pr, 2006.

\bibitem[Cer00]{Cer00}
Rapha{\"e}l Cerf, \emph{Large deviations for three dimensional supercritical
  percolation}, Ast\'erisque (2000), no.~267, vi+177.

\bibitem[Cer06]{Cer06}
R.~Cerf, \emph{The {W}ulff crystal in {I}sing and percolation models}, Lecture
  Notes in Mathematics, vol. 1878, Springer-Verlag, Berlin, 2006, Lectures from
  the 34th Summer School on Probability Theory held in Saint-Flour, July 6--24,
  2004, With a foreword by Jean Picard.

\bibitem[CI02]{CI02}
Massimo Campanino and Dmitry Ioffe, \emph{Ornstein-{Z}ernike theory for the
  {B}ernoulli bond percolation on {$\Bbb Z^d$}}, Ann. Probab. \textbf{30}
  (2002), no.~2, 652--682.

\bibitem[CIL10]{CIL10}
M.~Campanino, D.~Ioffe, and O.~Louidor, \emph{Finite connections for
  supercritical {B}ernoulli bond percolation in 2{D}}, Markov Process. Related
  Fields \textbf{16} (2010), no.~2, 225--266.

\bibitem[CN06]{CN06}
F.~Camia and C.~M. Newman, \emph{Two-dimensional critical percolation: the full
  scaling limit}, Comm. Math. Phys. \textbf{268} (2006), no.~1, 1--38.

\bibitem[CN07]{CM07}
\bysame, \emph{Critical percolation exploration path and {${\rm SLE}_6$}: a
  proof of convergence}, Probab. Theory Related Fields \textbf{139} (2007),
  no.~3-4, 473--519.

\bibitem[CP00]{CP00}
Rapha{\"e}l Cerf and {\'A}goston Pisztora, \emph{On the {W}ulff crystal in the
  {I}sing model}, Ann. Probab. \textbf{28} (2000), no.~3, 947--1017.

\bibitem[CS09]{CS09}
D.~Chelkak and S.~Smirnov, \emph{Universality in the {2D} {I}sing model and
  conformal invariance of fermionic observables}, Inv. Math. \textbf{189} (2012), no. 3, 515-580.

\bibitem[DC11]{Dum12}
H.~Duminil-Copin, \emph{Phase transition in random-cluster and {O}(n)-models},
  archive-ouverte.unige.ch/unige:18929 (2011), 360 pages, thesis.

\bibitem[DCGP11]{DGP12}
H.~Duminil-Copin, C.~Garban, and G.~Pete, \emph{The near-critical planar
  {FK-Ising} model}, to appear in Comm. Math. Phys., arXiv:1111.0144, 2011.

\bibitem[DCS11]{DCS11}
H.~Duminil-Copin and S.~Smirnov, \emph{Conformal invariance of lattice models},
  Lecture notes, in Probability and Statistical Physics in Two and More
  Dimensions (D.~Ellwood, C.~Newman, V.~Sidoravicius, and W.~Werner, eds.),
  CMI/AMS - Clay Mathematics Institute Proceedings, 2011.

\bibitem[DKS92]{DKS92}
R.~Dobrushin, R.~Koteck{\'y}, and S.~Shlosman, \emph{Wulff construction},
  Translations of Mathematical Monographs, vol. 104, American Mathematical
  Society, Providence, RI, 1992, A global shape from local interaction,
  Translated from the Russian by the authors.

\bibitem[GPS12]{GPS10a}
C.~Garban, G.~Pete, and O.~Schramm, \emph{Pivotal, cluster and interface
  measures for critical planar percolation}, JAMS to appear (2012),
  92.

\bibitem[GPS13]{GPSa}
\bysame, \emph{The scaling limits of dynamical and near-critical percolation},
  2013, arxiv:1305.5526, p.~86 pages.

\bibitem[Gri99]{Gri99}
G.~Grimmett, \emph{Percolation}, Springer Verlag, 1999.

\bibitem[IS98]{IS98}
Dmitry Ioffe and Roberto~H. Schonmann, \emph{Dobrushin-{K}oteck\'y-{S}hlosman
  theorem up to the critical temperature}, Comm. Math. Phys. \textbf{199}
  (1998), no.~1, 117--167.

\bibitem[Ken00]{Ken00}
R.~Kenyon, \emph{Conformal invariance of domino tiling}, Ann. Probab.
  \textbf{28} (2000), no.~2, 759--795.

\bibitem[Kes80]{Kes80}
H.~Kesten, \emph{The critical probability of bond percolation on the square
  lattice equals {${1\over 2}$}}, Comm. Math. Phys. \textbf{74} (1980), no.~1,
  41--59.

\bibitem[Kes82]{Kes82}
\bysame, \emph{Percolation theory for mathematicians}, Boston, 1982.

\bibitem[Kes87]{Kes87}
\bysame, \emph{Scaling relations for {$2$}{D}-percolation}, Comm. Math. Phys.
  \textbf{109} (1987), no.~1, 109--156.

\bibitem[MW73]{McCW73}
B.M. McCoy and T.T. Wu, \emph{The two-dimensional {I}sing model}, Harvard
  University Press, Cambridge, MA, 1973.

\bibitem[Nol08]{Nol08}
P.~Nolin, \emph{Near-critical percolation in two dimensions}, Electron. J.
  Probab. \textbf{13} (2008), no. 55, 1562--1623.

\bibitem[Pfi91]{Pfi91}
C.-E. Pfister, \emph{Large deviations and phase separation in the
  two-dimensional {I}sing model}, Helv. Phys. Acta \textbf{64} (1991), no.~7,
  953--1054.

\bibitem[Rus78]{Rus78}
L.~Russo, \emph{A note on percolation}, Z. Wahrscheinlichkeitstheorie und Verw.
  Gebiete \textbf{43} (1978), no.~1, 39--48.

\bibitem[Smi01]{Smi01}
S.~Smirnov, \emph{Critical percolation in the plane: conformal invariance,
  {C}ardy's formula, scaling limits}, C. R. Acad. Sci. Paris S\'er. I Math.
  \textbf{333} (2001), no.~3, 239--244.

\bibitem[Smi10]{Smi10}
\bysame, \emph{Conformal invariance in random cluster models. {I}.
  {H}olomorphic fermions in the {I}sing model}, Ann. of Math. (2) \textbf{172}
  (2010), no.~2, 1435--1467.

\bibitem[SS11]{SS11}
Oded Schramm and Stanislav Smirnov, \emph{On the scaling limits of planar
  percolation}, Selected works of {O}ded {S}chramm. {V}olume 1, 2, Sel. Works
  Probab. Stat., Springer, New York, 2011, With an appendix by Christophe
  Garban, pp.~1193--1247.

\bibitem[SW78]{SW78}
P.~D. Seymour and D.~J.~A. Welsh, \emph{Percolation probabilities on the square
  lattice}, Ann. Discrete Math. \textbf{3} (1978), 227--245, Advances in Graph
  Theory (Cambridge Combinatorial Conf., Trinity College, Cambridge, 1977).

\bibitem[SW01]{SW01}
S.~Smirnov and W.~Werner, \emph{Critical exponents for two-dimensional
  percolation}, Math. Res. Lett. \textbf{8} (2001), no.~5-6, 729--744.

\bibitem[Wer07]{Wer07}
W.~Werner, \emph{{\it Lectures on two-dimensional critical percolation}}, IAS
  Park City Graduate Summer School, 2007.

\end{thebibliography}

\begin{flushright}
\footnotesize\obeylines
  \textsc{Universit\'e de Gen\`eve}
  \textsc{Gen\`eve, Switzerland}
  \textsc{E-mail:} \texttt{hugo.duminil@unige.ch}
\end{flushright}
%\begin{flushright}
%\footnotesize\obeylines
%  \textsc{Universit\'e de Gen\`eve}
%  \textsc{Gen\`eve, Switzerland}
%  \textsc{E-mail:} \texttt{hugo.duminil@unige.ch}
%  \end{flushright}
\end{document}